\documentclass[12pt]{amsart}

\usepackage{amsthm,amsmath,amssymb,amsfonts,color}
\usepackage{mathrsfs}
\usepackage{tikz}
\usepackage{enumerate}
\usepackage{mathtools}
\usepackage{url}
\usepackage{hyperref}
\definecolor{myblue}{RGB}{25,25,112}

\newtheorem{thm}{Theorem}[section]
\newtheorem{lem}[thm]{Lemma}
\theoremstyle{definition}
\newtheorem{defn}[thm]{Definition}

\newtheorem{rmk}[thm]{Remark}

\DeclareMathOperator{\End}{End}
\DeclareMathOperator{\Aut}{Aut}
\newcommand{\Z}{\mathbb{Z}}
\newcommand{\C}{\mathbb{C}}
\newcommand{\Q}{\mathbb{Q}}
\newcommand{\B}{\mathbf{B}}
\newcommand{\E}{\mathbf{E}}

\begin{document}	
\title[Zeta Functions of Table Algebra Orders]{Computing zeta functions of table algebra orders using local zeta integrals }
\author[ ]{Angelica Babei${^*}$ and Allen Herman$^{\dagger}$}
\address{Department of Mathematics and Statistics, McMaster University, Hamilton, Ontario, L8S 4L8, Canada}\email{babeia@mcmaster.ca}
\address{Department of Mathematics and Statistics, University of Regina,
  Regina, Saskatchewan S4S 0A2, Canada}\email{Allen.Herman@uregina.ca}
\thanks{$^*$ The work of the first author was supported by a Simons Collaboration Grant (550029, to John Voight).} 
\thanks{$^{\dagger}$ The work of the second author is supported by an NSERC Discovery Grant.} 
\thanks{Data availability statement: This manuscript has no associated data.}
\begin{abstract}	
We investigate Solomon's zeta function for orders in the special case of orders generated by the standard basis of an integral table algebra, a special case of which is the integral adjacency algebra of an association scheme.  As Solomon's elementary method for computing this zeta function runs into computational difficulties for ranks $3$ or more, a more efficient method is  desired.  We give several examples to illustrate how the local zeta integral approach proposed by Bushnell and Reiner can be applied to compute explicit zeta functions for these orders.   
\end{abstract}

\maketitle

\section{Introduction}

\medskip
Let $\B$ be an integral basis of a finite-dimensional semisimple algebra $\Q \B$; that is, the structure constants relative to the basis $\B$ are integers.  In this case $\Z \B$ is a $\Z$-order in $\Q \B$.  Solomon's zeta function 
$$\zeta_{\Z \B}(s) = \sum_{n \ge 1} a_n n^{-s}$$ 
is the generating function for the sequence $a_n$ which counts the number of $\Z\B$-submodules of the regular module $\Z\B$  of a given index $n$ \cite{Solomon77}.  Among the many versions of global zeta functions for rings (see \cite{Rossmann2018} for an overview), Solomon's can be regarded as the ideal zeta function of the ring $\Z \B$, or, alternatively, the submodule zeta function for the regular module of $\Z \B$.  Since $\B$ is finite, in commutative cases $\Z\B$ will be isomorphic to the coordinate ring of an integral variety, and $\zeta_{\Z\B}(s)$ will be equal to the Hasse-Weil zeta function associated to this variety, thus an elementary case of an arithmetic zeta function.  Indeed, many of the interesting properties conjectured to hold for arithmetic zeta functions were previously shown for these zeta functions (see \cite{Bushnell-Reiner1980}). 

Convergence of the zeta function of a $\Z$-order for $Re(s)>1$ is a consequence of Hey's formula for the zeta function of a maximal order (see \cite[p. 307]{Solomon77}), for non-maximal orders this was shown by Jenner \cite{Jenner1963}.  In this situation Solomon proved the Euler product identity $\zeta_{\Z \B}(s) = \prod_p \zeta_{\Z_p \B}(s)$, where $p$ runs over the rational  primes, and showed that each factor $\zeta_{\Z_p\B}(s)$ is equal to $\zeta_{\Gamma_{p}}(s)$ times the value of a polynomial in $\Z[x]$ evaluated at $p^{-s}$, where $\Gamma_{p}$ is a maximal order of $\Q_p\B$ containing $\Z_p\B$.   For all but finitely many primes $p$, $\Z_p\B$ will be equal to $\Gamma_{p}$.  The exceptions are the primes $p$ that divide the discriminant of a global maximal order $\Gamma$ of $\Q \B$ or the least positive integer $f$ with $f \Gamma \subset \Z\B$ \cite{Solomon77}.  We will call these {\it relevant} primes.  

Solomon considered the case where $\B$ is the group basis of an integral group ring, and computed the zeta function for the integral group ring of a cyclic group of prime order $p$.  In this case $p$ is the only relevant prime, and the $p$-adic zeta function of $\Z_pC_p$ was obtained by using the Hermite normal form to count the ideals of $\Z_pC_p$.  This was extended to products of cyclic groups whose order is the product of two distinct primes by Hironaka (\cite{Hironaka81}, \cite{Hironaka85}), and to the elementary abelian $p$-group of order $p^2$ for a prime $p$ by Takegehara \cite{Takegahara87}.  Hanaki and Hirasaka calculated Solomon's zeta function for integral adjacency algebras of association schemes of prime order or rank $2$ in \cite{Hanaki-Hirasaka2016}, again using the Hermite normal form approach.   In analogy to Hironaka's work, Herman, Hirasaka, and Oh were able to extend this to the case of locally coprime tensor products \cite{Herman-Hirasaka-Oh2017}.  Hirasaka and Oh were able to use the Hermite normal form approach to calculate the zeta function of quotient polynomial rings $\Z[x]/p(x)\Z[x]$ where $p(x) \in \Z[x]$ has degree $3$ and factors completely over $\Z$ \cite{Hirasaka-Oh2018}.  Their work illustrates the complications arising in counting ideals using Hermite normal forms when $|\B|$ is $3$ or more.  

 Bushnell and Reiner developed a method for calculating the Solomon zeta function for a $\Lambda$-lattice, where $\Lambda$ is an arithmetic order in a semisimple algebra over a $p$-adic field that makes use of the zeta integrals introduced in Tate's thesis \cite{Bushnell-Reiner1980}.  Since then this approach has been used by Wittmann to compute submodule zeta functions for the $\Z C_p$-lattice $(\Z C_p)^n$ \cite{Wittmann2004} and by  Hofmann to compute the submodule zeta function for the $\Z$-lattice corresponding to a $\Z S_{n+1}$-module affording the irreducible character related to the hook partition $(2,1^{n-1})$ \cite{Hofman2016}.  The examples in this paper illustrate how to apply their technique to calculate, for any relevant prime $p$, the zeta function of $\Z_p \B$ where $\B$ is the $\Z$-basis consisting of the standard basis of a commutative integral table algebra, a situation which includes the cases where $\B$ is a finite abelian group or the set of adjacency matrices of a commutative association scheme as special cases.  We apply  this technique to calculate local zeta functions for the integral adjacency algebras of association schemes corresponding to complete graphs (Equation (\ref{completegraph})), the Petersen graph (Equation (\ref{Petersengraph})), the square (Theorem  \ref{Squaregraph}), and the generalized quadrangle of order $(2,1)$ (Theorem \ref{GenQuadranglegraph}).  The last two of these are rank $3$ cases that had not been done previously.  We  also give a formula for the zeta function of the adjacency algebra of crown graphs $K_{n,n}-I$ ($n$ odd) in dimension $4$ (Theorem \ref{crowngraph}). To compare our technique to the Hermite normal form approach, we will first review how to derive Equation  (\ref{completegraph}) with Hanaki and Hirasaka's Hermite normal form approach before re-calculating the same zeta function with with our method. 
 
 \section*{Acknowledgement}
 This version of the article has been accepted for publication, after peer
review, but is not the Version of Record and does not reflect post-acceptance
improvements, or any corrections. The Version of Record is available online at:
 \href{http://dx.doi.org/10.1007/s00009-023-02316-2}{\textbf{\color{myblue}dx.doi.org/10.1007/s00009-023-02316-2}}. Use of this Accepted Version is subject to the publisher’s Accepted
Manuscript terms of use \href{https://www.springernature.com/gp/open-research/policies/accepted-manuscript-terms}{\textbf{\color{myblue}www.springernature.com/gp/open-research/policies/accepted-manuscript-terms}}.

\section{The finite set of relevant primes}  

In all of our examples, $\B = \{b_0=1, b_1, \dots, b_{r-1} \}$ will be the standard basis of a commutative integral table algebra; in particular, 
\begin{itemize} 
\item the structure constants relative to $\B$ are nonnegative integers, 
\item the complex algebra $\C \B$ is a commutative semisimple algebra with skew-linear involution $*$, 
\item $\B$ is $*$-closed, 
\item for all $i \in \{0,1,\dots,r-1\}$ the coefficient of $1$ in the product $b_i b_i^*$ is the Perron-Frobenius eigenvalue $k_i$ of the regular matrix of $b_i$ in the basis $\B$, and 
\item if $j \in \{0,1,\dots,r-1\}$ and $b_j \ne b_i^*$, then the coefficient of $1$ in $b_ib_j$ is $0$.  
\end{itemize}
The Perron-Frobenius eigenvalue $k_i$ of the regular matrix of $b_i$ is called the {\it degree} of $b_i$ for all $i \in \{0,1,\dots,r-1\}$.  Since $b_0=1$ we have $k_0=1$.  The linear mapping of $\C\B$ given by $b_i \mapsto k_i$ is a one-dimensional algebra representation of $\C\B$ called the degree map.  (This notation is standard for table algebras, see \cite{Blau09} for background on table algebras and to see how integral adjacency algebras of association schemes and integral group rings fit into this framework.)  

The $\Z$-algebra $\Z \B$ is a $\Z$-order in the $r$-dimensional algebra $\Q \B$.  As a ring, $\Z \B$ is indecomposable; for standard integral table algebra bases, indecomposability of $\Z \B$ follows from the fact that the coefficient of $b_0$ in any nontrivial idempotent of $\Q\B$ is a rational number strictly between $0$ and $1$ \cite[Proposition 3]{Herman-Singh2018}.  When $\Z \B$ is commutative, $\Z \B$ will be properly contained in the maximal order $\Z \E$ of $\Q \B$ whose $\Z$-basis is the set of primitive idempotents $\E = \{e_i: i=0,\dots,r-1\}$ of $\Q \B$.  By convention, $e_0$ will be the primitive idempotent corresponding to the degree map. We will make use of the character formula for integral table algebras, which we now review from the sources \cite{AFM} and \cite{Blau09}.  If $\chi_i$ is the irreducible character corresponding to the primitive idempotent $e_i$, the values $\chi_i(b_j)$ are precisely the eigenvalues of the regular matrix of $b_j$ for $i \in \{0,1,\dots,r-1\}$, and the orthogonality of the idempotents $e_i$ imposes orthogonality relations that these character table entries must satisfy. The constant $n = \sum_{j=0}^{r-1} k_j$ is called the {\it order} of the table algebra.  The standard feasible trace of the algebra $\C\B$ is the linear map $\rho: \C\B \rightarrow \C$ given by $\rho(b_j) = n \delta_{j0}$ for $j \in \{0,1,\dots,r-1\}$, this is a positive linear combination of the irreducible characters $\chi_i$.  The {\it multiplicity} $m_i$ is the coefficient of $\chi_i$ in the expression $\rho = \sum_{i=0}^{r-1} m_i \chi_i$, the multiplicity $m_0$ of the degree map $\chi_0$ is always $1$.  The character formula for standard integral table algebras tells us for all $i \in \{0,1,\dots,r-1\}$
$$ e_i = \frac{m_i}{n} \sum_{j=0}^{r-1} \frac{\chi_i(b_j^*)}{k_j} b_j. $$
In the case where $\B$ is realized as the set of adjacency matrices of an association scheme, the multiplicities $m_i$ are positive integers.  

To determine our sets of relevant primes for a commutative integral table algebra $\Z\B$, we need two facts.  First, the maximal order $\Gamma = \Z \E$ has a $\Z$-basis of primitive idempotents, so its discriminant is $1$.  Second, it follows from the character formula that the least positive integer $f \coloneqq f(\Gamma:\Z\B)$ satisfying $f \Gamma \subseteq \Z \B$ will be a divisor of the {\it Frame number} 
$$\mathcal{F}(\B) = \dfrac{n^rk_1k_2\cdot \cdot \cdot k_{r-1}}{m_1\cdot \cdot \cdot m_r},$$
which is always a positive integer \cite[Lemma 3.7]{AFM}. Finally, we remark here that in all of the examples we consider in this paper, the commutative semisimple algebra $\Q \B$ will be split; i.e. $\Q \B \simeq \oplus_{i=0}^{r-1} \Q e_i$.  This is only used here when we identify $\Z_p \B$-lattices in $\Q \B$ up to isomorphism, it does not affect the set of relevant primes.

\section{ Comparing calculations in the case $|B|=2$ }\label{dim2}

When $\B = \{ b_0,b_1 \}$ is the standard basis of an integral table algebra of order $n$, then this algebra is realized up to exact isomorphism by the standard representation of the association scheme corresponding to the complete graph $K_n$ on $n$ vertices.  The standard matrices for this association scheme are $\{ I_n, J_n-I_n \}$, where $I_n$ is the $n \times n$ identity matrix and $J_n$ is the $n \times n$ matrix with all entries equal to $1$.  This implies $b_1^2 = (n-1)b_0 + (n-2)b_1$ and $b_1$ has minimal polynomial $\mu(x) = (x - (n-1))(x+1)$.  

We have $\Q\B = \Q e_0 + \Q e_1$, where $e_0 = \frac{1}{n}(b_0 + b_1)$ and $e_1 = 1 - e_0 = \frac{1}{n}((n-1)b_0-b_1)$.  Therefore, $\Gamma = \Z e_0 \oplus \Z e_1$ and the relevant primes are the divisors of $f(\Gamma:\Z\B)=n$. 

For a prime $p$, let $v_p$ denote the $p$-adic valuation on $\Q_p$; so for all $a \in \Q_p \setminus \{0\}$ we have that $a = p^{v_p(a)}a_0$ where $a_0 \in \Z_p \setminus p\Z_p$ is an element of valuation $v_p(a_0)=0$, and $v_p(0)=\infty$.  The function 
$$\zeta_{\Z_p}(s) = (1-p^{-s})^{-1} = \displaystyle{\sum_{i=0}^{\infty}} p^{-is}$$ 
of a complex variable $s$ will denote the (Dedekind) zeta function of the ring of $p$-adic integers. 

\smallskip
\quad (a) {\bf Using the Hermite normal form approach.} The zeta function of the integral adajacency algebra $\Z K_n$ corresponding to the complete graph association scheme $K_n$ was computed using this method by Hanaki and Hirasaka \cite{Hanaki-Hirasaka2016} and extended in \cite{Herman-Hirasaka-Oh2017} to the zeta function of $\mathcal{O}K_n$ where  $\mathcal{O}$ is the ring of integers in an algebraic number field.  

Assume $p$ is a prime divisor of $n$, and let $m=v_p(n)$.  For convenience we change the basis to $\{b_0,b_0+b_1\}$ as $(b_0+b_1)^2= n(b_0+b_1)$ is easier to work with.  The $\Z_p$-sublattices $I$ of $\Z_p\B$ with index $p^m$ are parametrized by their Hermite normal form relative to our ordered basis of $\Z_p \B$.   Suppose $I$ is a $\Z_p$-sublattice of $\Z_p\B$ of finite index.  We first choose $\gamma b_0 + \delta (b_0+b_1) \ne 0$ in $I$ with $r_1 = v_p(\gamma)$ minimal.  Adjusting by a unit we can assume this element is of the form $p^{r_1}b_0+\delta (b_0+b_1)$.  Then choose $r_2$ minimal such that $p^{r_2}(b_0+b_1) \in I-\{0\}$.  If $a$ is the remainder of $\delta$ divided by $p^{r_2}$ then $p^{r_1}b_0 + a (b_0+b_1) \in I$.  We claim that $I$ is generated by $p^{r_1}b_0 + a(b_0+b_1)$ and $p^{r_2}(b_0+b_1)$, and so the index of $I$ in $\Z_p\B$ is $p^{r_1+r_2}$.   If $\gamma b_0 + \delta b_1 \in I$, then $v_p(\gamma) \ge r_1$, so subtracting a multiple of $p^{r_1}b_0 + a(b_0+b_1)$ leaves $\epsilon (b_0+b_1)$, for some $\epsilon \in \Z_p$, and by the choice of $r_2$, this $\epsilon$ is a multiple of $p^{r_2}$, proving our claim.  

To be an ideal of $\Z_p\B$, $I = \Z_p[(p^{r_1}b_0 + a(b_0+b_1)] + \Z_p[p^{r_2}(b_0+b_1)]$ must be closed under multiplication by $(b_0+b_1)$, which leads to the condition $p^{r_2}$ divides $p^{r_1} + an$.  For a given ideal $I$ of index $p^{r_1+r_2}$ we need to count the number of choices of $a$ mod $p^{r_2}$ with this property.  

If $r_1 \ge r_2$ and $r_2 < v_p(n)=m$, then  for every choice of $a$ mod $p^{r_2}$, we have $p^{r_1}+an \in p^{r_2}\Z_p$, so we get $p^{r_2}$ solutions.  If $r_1 < r_2 < m$, then $p^{r_1}+an$ will not be a multiple of $p^{r_2}$, so there are no solutions.  So now we can assume $r_2 \ge m$.  If also $r_1 \ge m$, then $p^{r_1}+an = \alpha p^{r_2}$ implies $a = n_0^{-1}(-p^{r_1 - m} + \alpha p^{r_2 - m})$.  So we get a unique solution for $a$ mod $p^{r_2-m}$ and hence $p^{m}$ solutions for $a$ mod $p^{r_2}$.  On the other hand, if $r_1 < m \le r_2$, then $a = n_0^{-1}(-p^{r_1-m} + \alpha p^{r_2-m})$ is not in $\Z_p$ for any choice of $\alpha \in \Z_p$, so there are no solutions. 

Therefore, 
$$\begin{array}{rcl} 
\zeta_{\Z_p\B}(s) &=& \displaystyle{\sum_{r_2=0}^{m-1}} \displaystyle{\sum_{r_1=r_2}^{\infty}} p^{r_2} p^{(r_1+r_2)(-s)} + \displaystyle{\sum_{r_2=m}^{\infty}} \displaystyle{\sum_{r_1=m}^{\infty}} p^{m} p^{(r_1+r_2)(-s)} \\
&=& \displaystyle{\sum_{r_2=0}^{m-1}} p^{r_2(1-s)} \displaystyle{\sum_{r_1=r_2}^{\infty}} p^{r_1(-s)} + p^{m} \displaystyle{\sum_{r_2=m}^{\infty}} p^{r_2(-s)} \displaystyle{\sum_{r_1=m}^{\infty}} p^{r_1(-s)} \\
&=& \displaystyle{\sum_{r_2=0}^{m-1}} p^{r_2(1-2s)} (1-p^{-s})^{-1} + p^{m(1-2s)} (1-p^{-s})^{-2} \\
&=& \bigg[ \bigg( \displaystyle{\sum_{r_2=0}^{m-1}} p^{r_2(1-2s)}(1-p^{-s}) \bigg) +p^{m(1-2s)} \bigg] \zeta_{\Z_p}(s)^2.
\end{array}$$

So the zeta function of $\Z K_n$ is 
\begin{equation}\label{completegraph}
\zeta_{\Z K_n}(s) = \bigg[ \prod_{p | n} \bigg( \displaystyle{ (1-p^{-s})\sum_{r=0}^{v_p(n)-1}} p^{r(1-2s)}  +p^{v_p(n)(1-2s)} \bigg) \bigg] \zeta_{\Z}(s)^2. 
\end{equation}
 
\begin{rmk} {\rm The polynomial factor in the above zeta  function is in agreement with the observation Bushnell and Reiner established in the case where the order $\Z_p\B$ is the group ring of a finite group over $\Z_p$: if $m=v_p(f)$, then there exists $g(x) \in \Z [x]$ with constant term $1$ and leading term $p^m x^{2m}$ such that $\zeta_{\Z_p \B}(s) = g(p^{-s}) \zeta_{\Z_p}(s)^r$ \cite[pg. 158]{Bushnell-Reiner1980}.  Their observation is a consequence of a strong form of the local functional equation they established for the zeta function of orders that are stabilized by a suitable trace form.  As we will see, this property need not hold for integral table algebras.  Any standard integral table algebra basis $\B$ that is not a group will not be self-dual with respect to the normalized trace form because the dual to $b_i$ is $\frac{1}{k_i}b_i$. }
\end{rmk}

\smallskip
\quad (b) {\bf Using Bushnell and Reiner's formula for local zeta integrals.} We summarize the method in \cite{Bushnell-Reiner1980}. Let $A$ be a finite dimensional semisimple $\Q$-algebra,  $\Lambda$ a $\Z$-order in $A$, and $V$ a finitely-generated left $A$-module. Fix some full $\Lambda$-lattice $L$ in $V$, and define the Solomon zeta function \[ \zeta_\Lambda(L; s)=  \sum_{N \subseteq L} (L:N)^{-s},\] where the sum is over all full $\Lambda$-lattices $N$ in $L$. When $L=\Lambda$, we get $\zeta_\Lambda(\Lambda; s) = \zeta_\Lambda(s)$, and we drop the $\Lambda$ in the parentheses. Given another $\Lambda$-lattice $M$ in $V$, the  {\it genus} of $M$ consists of all full $\Lambda$-lattices $N$ in $V$ such that there exists a $\Lambda_p$-isomorphism $N_p\cong M_p$  at each finite prime $p$. We write $N \sim M$ if $N$ is in the genus of $M$. Define  the genus zeta function 
 \[ Z_\Lambda(L, M; s) = \sum_{\substack{N \subseteq L\\ N \sim M}} (L : N)^{-s}, \] 
where the sum is over all full $\Lambda$-lattices $N \subseteq L$ such that $N$ lies in the genus of $M$.

By the Jordan-Zassenhaus theorem, there are only finitely many genera of full $\Lambda$-lattices in $V$, so we can pick a finite set $\mathscr{G}$ of genus representatives. Then 
\[ \zeta_\Lambda(L; s)= \sum_{M \in \mathscr{G}} Z_\Lambda(L, M ; s).\]

Now we turn to the local case, where $A$ is a $\Q_p$-algebra, $\Lambda$ is a $\Z_p$-order in $A$, and $V$, $L$ and $M$ are the local analogues of the setting above.  Since now only one prime is involved,
 \[ Z_\Lambda(L, M; s)=\sum_{\substack{N \subseteq L\\ N \cong M}} (L : N)^{-s}, \] 
and the sum is over all full sublattices $N$ of $L$ such that $N \cong M$ as $\Lambda$-lattices.  Therefore we can express $\zeta_\Lambda(L; s)=\sum_M Z_\Lambda(L, M; s)$, where the sum is over all isomorphism classes of full $\Lambda$-lattices in $V$. Moreover, Equation (11) in \cite{Bushnell-Reiner1980} gives

\begin{equation}
\label{genuszetaeqn}
 Z_\Lambda(L, M; s)=\mu(\Aut M)^{-1} (L:M)^{-s} \int_{B^\times}  \Phi_{\{M:L\}}(x) \left\Vert x \right\Vert_V^s d^\times x,
\end{equation}
 where $B=\End_A(V)$, $B^\times=\Aut_A(V)$, $\displaystyle (L:M)=\frac{(L:(L \cap N))}{(M:(L \cap M))}$,  $ \Phi_{\{M:L\}}$ is the characteristic function in $B$ of the lattice 
\[\{M:L\}=\{ x \in B \, | \, Mx \subseteq L\},\]
the norm $\left\Vert x \right\Vert_V=(Nx:N)$ where $N$ is any full $\Z_p$-lattice in $V$, and $d^\times x$ is a multiplicative Haar measure on $B^\times$ that has been normalized so that $\int_{\Lambda_{0}^\times}d^\times x=1$ when $\Lambda_0$ is a maximal order of $B$.   We note that $\mu(\Aut M)=\mu(\{ M: M\}^\times)$, and when $M$ is a $\Z_p$-order,   $\mu(\Aut M)=\mu (M^\times)$.

In our (local) calculations for $\zeta_\Lambda(s)$, we obtain representatives of the isomorphism classes of $\Lambda$-lattices $M$ for the genus zeta functions $Z_\Lambda(\Lambda, M; s)$ by considering $\Lambda$-lattices lying between $\Lambda$ and the  maximal order in $A$.

\begin{lem} 
\label{LatticesAbove}
Suppose $A$ is a split semisimple commutative algebra over  a local field $K$ with finite residue field. Let $R$ be the ring of integers of $K$, and let $\pi R$ be its maximal ideal.  Let $\Lambda_0$ be the maximal $R$-order in $A$, and let $\Lambda$ be any $R$-order in $A$.  Then every full $\Lambda$-lattice $N$ in $A$ is isomorphic as a $\Lambda$-lattice to  a full $\Lambda$-lattice $M$ satisfying  $\Lambda \subseteq M \subseteq \Lambda_0$.  In addition, given any $\Lambda$-isomorphism $\phi: M \rightarrow M'$, where $M$ is an $R$-order and $M, M'$ are full $\Lambda$-lattices with $\Lambda \subseteq M \subseteq \Lambda_0$ and $ \Lambda \subseteq  M' \subseteq \Lambda_0$,  implies $M=M'$.
\end{lem}

\begin{proof}
By multiplying by an appropriate power of $\pi$, we can assume $N \subseteq \Lambda_0$ and $\pi^{-1} N \not\subseteq \Lambda_0$.  We claim that such a lattice is isomorphic to one that contains a unit of $\Lambda_0$.  Let $\mathcal{E} = \{e_1,\dots,e_r\}$ be set of the primitive idempotents of $A$; with our assumptions this set is an $R$-basis for $\Lambda_0$.  Suppose that there exists an $e_i \in \mathcal{E}$ such that for all $n \in N$, $ne_i \in \pi R e_i$.  Let $\phi_i : N \rightarrow \Lambda_0$ be given by $\phi_i(n) = n(1-e_i) + \pi^{-1}ne_i$, for all $n \in N$.  Since $\phi_i$ is an injective $\Lambda_0$-linear map, it is a $\Lambda$-lattice isomorphism.  After applying as many of these maps as we can, we will obtain a $\Lambda$-lattice $N' \subseteq \Lambda_0$ that is isomorphic to $N$ with the property that for all $e_i \in \mathcal{E}$, $N'e_i \not\subseteq \pi R e_i$. In particular, we have $\pi R e_i \subsetneq N'e_i \subseteq Re_i$, and for such an $R$-lattice, we can take linear combinations of elements of $N'$ to obtain an element $u \in N'$ with the property that $ue_i \not\in \pi R e_i$ for all $e_i \in \mathcal{E}$, and such an element is a unit of $\Lambda_0$.  
It follows that $M = N'u^{-1}$ will be a $\Lambda$-lattice isomorphic to $N$ and contains $1$.  Since $1 \in M$ and $N' \subseteq \Lambda_0$, we have $\Lambda =\Lambda \cdot 1 \subseteq M \subseteq \Lambda_0$, and $M$ is a $\Lambda$-lattice with the desired properties. 

To prove the last claim, suppose we have an $\Lambda$-lattice isomorphism $\phi: M \rightarrow M'$, such that $\Lambda \subseteq M\subseteq \Lambda_0 $, $\Lambda \subseteq  M'\subseteq \Lambda_0$, and $M$ is an $R$-order. Then $\phi$ extends to an $A$-linear map $KM \cong KM'$, so it is given by multiplication by a unit $u \in A^\times$. In particular, $u=\phi(1) \in \Lambda_0$. In addition, since $1 \in M'$ and $\phi(u^{-1})=1$, we have $u^{-1} \in M$.  Since $M$ is an order,  $R[u^{-1}] \subseteq M$, and since $u \in R[u^{-1}]$ by \cite[(2) on pg. 208]{Kleinert1994}, we get  $u \in M^\times$. Therefore, $M'=M u = M$.
\end{proof}

In our specific examples, $K=\Q_p$ and $R=\Z_p$. Then each $\Z_p$-lattice $M$ has a basis in Hermite  normal form whose coefficients are rational \cite[Theorem 3.1.7]{Caruso}. By  exhausting all possible Hermite normal forms of $\Lambda$-lattices lying between $\Lambda$ and $\Lambda_0$, Lemma \ref{LatticesAbove} gives representatives for all the isomorphism classes of $\Lambda$-lattices. When these lattices are also $\Z_p$-orders, the lemma guarantees that they are nonisomorphic. To handle the case when we obtain some $\Lambda$-lattices which are not  $\Z_p$-orders, we examine their Hermite normal forms in more detail as follows. 

\begin{defn}
Let $M$ be an $\Z_p$-lattice in $A$ with basis $\{v_i\}_{i=0}^{r-1}$ in Hermite normal form with respect to the basis $\{e_i\}_{i=0}^{r-1}$; in particular, $v_i=\sum_{j=i}^{r-1} v_{ij}  e_j$ with $v_{ij} \in   \Q_p$. We define the \textit{diagonal} of the lattice $M$ to be the tuple $(v_{00}, v_{11}, \dots, v_{r-1, r-1})$, defined up to multiplication by elements of $\Z_p^\times$.
\end{defn}

\begin{rmk}
\label{HermiteZeroes} 
Suppose we have two isomorphic $\Lambda$-lattices $M, M'$ with $\Lambda \subseteq M, M' \subseteq \Lambda_0$. As in the proof of the last claim in Lemma \ref{LatticesAbove}, $\phi: M \rightarrow M'$ is given by multiplication by $u \in \Lambda_0$, where also $u^{-1} \in M \subseteq \Lambda_0$. Therefore, $u \in \Lambda_0^\times$, and since $\Lambda_0^\times= \oplus_{i=0}^{r-1} \Z_p^\times e_i$, we can write $u=\sum_{i=0}^{r-1} u_i e_i$ with $u_i \in \Z_p^\times$.

Then one can easily check the following:
\begin{enumerate}
    \item If $B=\{v_0, v_1, \dots, v_{r-1}\}$ is an upper triangular basis of $M$ (i.e.  $v_i=\sum_{j=i}^{r-1} v_{ij} e_j$ ), the image $\phi(B)=\{\phi(v_0), \phi(v_1), \dots, \phi(v_d)\}$ is also upper triangular.
    \item If $B$ is is upper triangular and some vector $v_k =\sum_{j=k}^{r-1} v_{kj} e_j \in B$ has one or more consecutive zero coefficients $v_{k,k+1}=\dots=v_{k,k+\ell}=0$ for  $1 \le \ell \le r-1-k$, then the basis $\{w_i\}$ of $\phi(M)$ with $w_i=\sum_{j=i}^{r-1} w_{ij} e_j$ in Hermite form also has  $w_{k,k+1}=\dots= w_{k,k+\ell}=0$. This follows since Hermite reduction of $\phi(B)$ will only affect the zero entries to the right of $w_{k,k+\ell}$. 
    \item The diagonals of $M$ and $\phi(M)$ are equal. If $M$ has diagonal $(v_{00}, v_{11}, \dots, v_{r-1, r-1})$, then $\phi(M)$ has diagonal $(u_0v_{00}, u_1v_{11}, \dots, u_{r-1} v_{r-1, r-1})$, where $u_j \in \Z_p^\times$. Therefore, up to multiplication by elements in $\Z_p^\times$, the two diagonals are the same.
\end{enumerate}
\end{rmk}

 \textbf{Recalculation of the complete graph zeta function using local zeta integrals.} We will first introduce some notation.  Consider a semisimple commutative $\Q_p$-algebra $A$ of arbitrary rank $r$. First, we can denote $\Z_p$-lattices in $A$ by their bases. For example, given a  basis $\{v_0, v_1, \dots,v_{r-1}\}$ of a $\Z_p$-lattice $L$, we write  $L=\langle v_0, v_1, \dots, v_{r-1}\rangle$.

Furthermore, once we fix the $\Z_p$-basis of the maximal order $\Lambda_0=\langle v_0, v_1, \dots, v_{r-1}\rangle$, we will denote the $\Z_p$-lattice $\langle p^{n(0)}v_0, p^{n(1)}v_1, \dots, p^{n(r-1)}v_{r-1} \rangle$ by $\Lambda_{n(0), \dots,  n(r-1)}$. In particular, the maximal order is denoted by $\Lambda_{0,0,\dots, 0}$. Also  note that $p^k \Lambda_{n(0), \dots, n(r-1)}=\Lambda_{n(0)+k,  \dots, n(r-1)+k}$ for all $k \in \Z$. 

When $\B$ is the standard basis of an integral table algebra of rank $2$, then we have $\displaystyle A=\Q_p\B \cong \Q_pe_0 \oplus \Q_pe_1$, our order is $\Lambda=\Z_p\B$,  and we fix the basis for the maximal order  $\Lambda_{00} = \langle e_0+e_1, e_1\rangle$.  Let $m =  v_p(n)$. Since $\displaystyle A \cong \frac{\Q_p[x]}{(x-(n-1))} \oplus \frac{\Q_p[x]}{(x+1)}$, we can represent $\Lambda$ as 
\[ \Lambda =\{ (a, b) \in \Z_p^2: a \equiv b \pmod{p^m}\}.\]

 Therefore, $\Lambda=\Lambda_{0, m}=\langle e_0+e_1, p^m e_1\rangle$, where the basis is in Hermite normal form.  The lattices $M$ with $\Lambda \subseteq M \subseteq \Lambda_{00}$ have bases in Hermite normal form $\{ e_0+e_1, p^j e_1\}$ for $0 \le j \le m$. Therefore, by Lemma \ref{LatticesAbove}, the tower of overorders of $\Lambda$   
 \[\Lambda_{00} \supsetneq \Lambda_{01} \supsetneq \dots \supsetneq \Lambda_{0m}=\Lambda\] gives a complete set of representatives of non-isomorphic $\Lambda$-modules in $V=A$. This gives 
\[ \zeta_{\Z_p\B}(s) =  \zeta_\Lambda(\Lambda; s)=  \sum_{0 \le i \le m} Z_{\Lambda}(\Lambda, \Lambda_{0i} ; s).\]

 Moreover, since $V=A$ and $A = \End_A(V)$,  Equation (\ref{genuszetaeqn}) in this case gives
\[ Z_\Lambda(\Lambda, \Lambda_{0i}; s)=\mu(\Lambda_{0i}^\times)^{-1} (\Lambda: \Lambda_{0i})^{-s} \int_{A^\times}  \Phi_{\{\Lambda_{0i}:\Lambda\}}(x) \left\Vert x \right\Vert_V^s d^\times x,\] where  $d^\times x$ is the multiplicative Haar measure on $A^\times$ such that $\int_{\Lambda_{00}}d^\times x=1$.

For each $0 \le i \le m$,  the index $(\Lambda:\Lambda_{0i})$ is equal to $p^{i-m}$. The measure $\mu(\Lambda_{0i}^\times)$ is given by  $\mu(\Lambda_{0i}^\times)=(\Lambda_{00}^\times:\Lambda_{0i}^\times)^{-1}$. In particular, $\mu(\Lambda_{00}^\times)=1$ and $\mu(\Lambda_{0i}^\times)=(p^i-p^{i-1})^{-1}$  when $i \ge 1$. 

Using standard linear algebraic techniques we obtain 
$$\{\Lambda_{0i}:\Lambda\} =\langle p^{m-i}(e_0+e_1), p^me_1\rangle =\Lambda_{m-i, m}=p^{m-i}\Lambda_{0i}.$$

We would like to decompose each $\Lambda_{m-i, m}$ as a disjoint union of sets over which we can compute the appropriate $p$-adic integrals. Let $p^{m-i}\alpha \cdot 1 +p^m\beta e_1 \in\Lambda_{m-i, m}$ for some $\alpha, \beta \in \Z_p$. Then either $\alpha=p\alpha_1$, or $\alpha=p\alpha_1+u$ for some $u \in  \{1, \dots, p-1\}$, and $\alpha_1\in \Z_p$. In the first case, $p^{m-i+1}\alpha_1\cdot 1 +p^m\beta e_1 \in \Lambda_{m-(i-1), m}=\{\Lambda_{0,i-1}: \Lambda\}$. In the second case, $p^{m-i+1} \alpha_1\cdot 1 + p^{m} \beta e_1 + p^{m-i} u \in p^{m-i} \cdot u (p \Lambda_{0,i-1}+1)$. Note that $p\Lambda_{0,i-1}+1$ is a subgroup of the unit group $\Lambda_{0,i-1}^\times$, and each $ u (p \Lambda_{0,i-1}+1)$ is a multiplicative translate. Therefore, 
\begin{align*}
 \int_{\Lambda_{m-i, m} \cap A^\times}  \left\Vert x \right\Vert_V^s d^\times x & =  &  \int_{\Lambda_{m-(i-1),m} \cap A^\times}  \left\Vert x \right\Vert_V^s d^\times x + (p-1)  \int_{p^{m-i} (p\Lambda_{0,i-1}+1)}  \left\Vert x \right\Vert_V^s d^\times x \\
 & = &   \int_{\Lambda_{ m-(i-1),m} \cap A^\times}  \left\Vert x \right\Vert_V^s d^\times x + (p-1) p^{-2(m-i)s} \mu(p\Lambda_{0,i-1}+1).
\end{align*}
We thus obtain a recursive formula for each integral, and the base step is computing $\int_{\Lambda_{m, m} \cap A^\times}  \left\Vert x \right\Vert_V^s d^\times x$, where as $\Z_p$-lattices, $\Lambda_{m,m} \cong p^m\Z_p \oplus p^m \Z_p$. Since $p^m\Z_p-\{0\}=\bigcup_{k=m}^\infty p^k\Z_p^\times$, we find that 
\begin{align*}
\displaystyle \int_{A^\times \cap(p^m\Z_p\oplus p^m\Z_p)} \left\Vert x \right\Vert_V^s d^\times x &= \left(\int_{p^m\Z_p-\{0\}} (\Z_pa:\Z_p)^s d^\times a\right)^2=\left(\sum_{k=m}^\infty  \int_{p^k\Z_p^\times} p^{-ks} d^\times a\right)^2 \\
&= p^{-2ms} (1-p^{-s})^{-2},
\end{align*}
and therefore
\begin{align*}
 \int_{\Lambda_{m-i,m} \cap A^\times}  \left\Vert x \right\Vert_V^s d^\times x & =  &  p^{-2ms} (1-p^{-s})^{-2} + (p-1) \sum_{j=1}^i  p^{-2(m-j)s} \mu(p\Lambda_{0, j-1}+1) .
\end{align*}

In order to find $\mu(p\Lambda_{0,j-1}+1)$, note that each element in $(p\Lambda_{0,j-1}+1)$ is of the form $(1+pb, 1+pb+p^j c)$ for $b, c \in \Z_p$. Therefore, there are $p^{j+1}$ possible such elements modulo $p^{j+1}$. On the other hand, any element in $\Lambda_{00}^\times$ is of the form $(u,v)$ with $u, v \in \Z_p^\times$, and therefore there are $(p^j(p-1))^2$ possible such elements modulo $p^{j+1}$.  Therefore, $\mu(p\Lambda_{0,j-1}+1)=p^{1-j}(p-1)^{-2}$, and 
\[ \displaystyle \int_{\Lambda_{ m-i,m} \cap A^\times}  \left\Vert x \right\Vert_V^s d^\times x =   p^{-2ms} (1-p^{-s})^{-2} + (p-1) \sum_{j=1}^i  \dfrac{p^{-2(m-j)s} p^{1-j}}{(p-1)^2} \]
\begin{equation}\label{Lambda_mm-i}
  =  p^{-2ms} (1-p^{-s})^{-2} + p^{-2ms} \sum_{j=1}^i  \dfrac{p^{2js-j+1}}{p-1}. 
\end{equation}

Therefore, $Z_\Lambda(\Lambda, \Lambda_{00} ; s)=p^{-ms}(1-p^{-s})^{-2}$, and for each $1 \le i \le m$, 
$$\begin{array}{rcl} 
Z_\Lambda(\Lambda, \Lambda_{0i} ; s)& = & \displaystyle   p^{i-1}(p-1) p^{(m-i)s}\left(p^{-2ms} (1-p^{-s})^{-2} + p^{-2ms} \sum_{j=1}^i  \frac{p^{2js-j+1}}{p-1}\right) \\
 & = &p ^{-ms-is+i-1}(p-1)\zeta_{\Z_p}(s)^2+p^{-ms-is+i}\sum_{j=1}^{i} p^{2js-j} \\
 & = &p ^{-ms-is+i-1}(p-1)\zeta_{\Z_p}(s)^2+p^{-ms+is}\sum_{j=1}^{i} p^{(i-j)-2(i-j)s}.
\end{array}$$

When we sum the formulas for our genus zeta functions to get $\zeta_{\Z_p\B}(s) = \sum_{i=0}^m Z_{\Lambda}(\Lambda,\Lambda_{0i};s)$, we get a formula that can be simplified to the one  in Equation (\ref{completegraph}).

\medskip
We end this section with an example of rank $4$ to illustrate how, on occasion, the indecomposable order $\Z\B$ can be locally decomposable at the relevant primes, which makes calculating its zeta function much easier. 

\medskip
(c) {\bf The crown graphs $K_{n,n}-I$, $n$ an odd integer.}  If $n \ge 4$, $K_{n,n}-I$ is the complete multipartite graph $K_{n,n}$ with a perfect matching removed.  It is a distance regular graph with intersection array $[n-1,n-2,1;1,n-2,n-1]$, which tells us that in the regular representation of the adjacancy algebra of the association scheme, $b_1$ is represented as the tridiagonal matrix 
$$b_1 = \begin{bmatrix} 0 & n-1 & 0 & 0\\ 1 & 0 & n-2 & 0 \\ 0 & n-2 & 0 & 1 \\ 0 & 0 & n-1 & 0 \end{bmatrix}.$$ 
(For readers unfamiliar with this concept, this intersection array convention applies to all distance regular graphs.  The numbers in the array give the lists of super- and sub-diagonal entries of $b_1$, the diagonal is filled in with the criteria that the row sum has to be $k_1$.  The remaining basis elements are determined inductively from this matrix, since we see that $b_i$ will be a polynomial of degree $i$ in $b_1$, for $i =2, \dots, r-1$.  It follows that the character table of the association scheme is entirely determined by the eigenvalues of $b_1$.  For details on how association schemes arise from the intersection arrays of distance regular graphs, we refer the readers to  \cite{Bannai-Ito84}.  Here we are using the database \cite{DRdotorg} to connect distance-regular graphs to their intersection arrays.)  

From the character table of this association scheme we find the primitive idempotents of $\Q\B$ are $e_0=\frac{1}{2n}(b_0+b_1+b_2+b_3)$, $e_1=\frac{n-1}{2n}b_0 - \frac{1}{2n}b_1 -\frac{1}{2n}b_2 + \frac{n-1}{2n}b_3$, $e_2 =\frac{1}{2n}(b_0-b_1+b_2-b_3),$ and $e_3 = \frac{n-1}{2n}b_0 + \frac{1}{2n}b_1 - \frac{1}{2n}b_2 - \frac{n-1}{2n} b_3.$  From this we can write the elements of $\B$ in terms of the primitive idempotents: 
\begin{align*} 
b_0 &= e_0+e_1+e_2+e_3{\color{blue} ,} \\
b_1 &=(n-1)e_0-(n-1)e_1-e_2+e_3{\color{blue} ,} \\
b_2 &=(n-1)e_0+(n-1)e_1-e_2-e_3, \mbox{ and } \\
b_3 &=e_0 - e_1 + e_2 -e_3. 
\end{align*} 
Elimination allows us to reduce to a more convenient basis of $\Z\B$: $\{b_0, 2e_1 + e_3, ne_2+ne_3, 2ne_3 \}$.  As we are assuming $n$ is odd, the relevant primes are $2$ and the prime divisors of $n$.   When we localize at $p=2$ we can reduce the basis of $\Z_2\B$ further to $\{e_0+e_1,2e_1,e_2+e_3,2e_3\}$.  Recall the calculation of the zeta function $\zeta_{\Z K_n}$ of the integral adjacency algebra $\Z K_n$ corresponding to the complete graph association scheme $K_n$ on $n$ vertices from subsection (b).  In the notation of (b), we see that $\Z_2\B$ is the direct sum $\Lambda_{01} \oplus \Lambda_{01}$, and so the zeta function will be
$$ \zeta_{\Z_2\B}(s) = \zeta_{\Z_2K_2}(s)^2. $$ % = (1 - 2^{-s} + 2^{1-2s})^2\zeta_{\Z_2}(s)^2$$ 
Similarly, when we localize at an odd prime $p$ dividing $n$ with $m=v_p(n)$, then the basis of $\Z_p\B$ reduces to $\{e_0+e_2,p^me_2, e_1+e_3, p^me_3\}$, so the zeta function will be 
$$ \zeta_{\Z_p\B}(s) = \zeta_{\Z_p K_{p^m}}(s)^2. $$ 

\begin{thm}\label{crowngraph}
For all positive odd integers $n > 1$, if $\B$ is the basis of adjacency matrices of the association scheme corresponding to the crown graph $K_{n,n}-I$, then 
$$ \zeta_{\Z\B}(s) =  \zeta_{\Z K_2}(s)^2 \zeta_{\Z K_n}(s)^2, $$ where $\zeta_{\Z K_n}$ is the  zeta function of the integral adjacency algebra $\Z K_n$ corresponding to the complete graph association scheme $K_n$ on $n$ vertices.
\end{thm}

\section{Cases when $|\B|=3$ }

In this section we compute $\zeta_{\Z\B}(s)$ for the integral adjacency algebras associated to some specific strongly regular graphs using the local zeta integral approach.  In each case the rank of $\B$ is $3$.  The graphs are: 
\begin{itemize} 
\item the Petersen graph (intersection array $[3,2;1,1]$); 
\item the square (intersection array $[2,1;1,2]$); and 
\item the generalized quadrangle $GQ(2,1)$ (intersection array $[4,2;1,2]$). 
\end{itemize} 

%In each case, we obtain the overorders of $\Z_p\B$ by first finding all Hermite normal forms of the finitely many $\Z_p$-lattices containing $\Z_p\B$, and then observing these are not isomorphic as $\Z_p\B$-lattices.  

\medskip
\noindent {\bf (a) The Petersen graph.} As computed in \cite{Hirasaka-Oh2018}, the zeta function of the integral adjacency algebra $\Z P$ of the association scheme corresponding to the Petersen graph $P$ is 
$$ \zeta_{\Z P}(s) = \bigg(\prod_{p=2,3,5} (p^{1-2s}-p^{-s}+1) \bigg) \zeta_{\Z}(s)^3.$$
We can also obtain the zeta function using the Bushnell and Reiner approach.  The Petersen graph is the distance regular graph with intersection array $[3,2;1,1]$, so the regular matrix for our basis element $b_1$ is
$$b_1 = \begin{bmatrix} 0 & 3 & 0 \\ 1 & 0 & 2 \\ 0 & 1 & 2 \end{bmatrix}.$$ 
After using this to generate the remaining basis elements and the character table of the association scheme, we find that the 
primitive idempotents of $\Q\B$ are $e_0 = \frac{1}{10}(b_0+b_1+b_2)$, $e_1=\frac{2}{5}(b_0-\frac{2}{3}b_1+\frac{1}{6}b_2)$, and $e_2=\frac{1}{2}(b_0+\frac{1}{3}b_1-\frac{1}{3}b_2)$.  From these formulas we see that $f(\Gamma: \Z\B)=30$, $b_0 = e_0+e_1+e_2$, $b_1=3e_0-2e_1+e_2$, and $b_2=6e_0+e_1-2e_2$.  

We can choose the basis $\{10 e_0, 5e_1+2e_2, e_0+e_1+e_2\}$ for the order $\Z\B$.  The primes at which the completion of $\Z\B$ is not maximal are $p=2,3,5$.  At $p=2$ the completion $\Z_p \B$ has basis $\{2e_0, e_1, e_0+e_2\}$, at $p=3$ it has basis $\{e_0, 3e_1, e_1+e_2\}$, and at $p=5$ it has basis $\{5e_0, e_2, e_0+e_1\}$. Because in each case the basis is of the form $\{pe_i, e_j, e_i+e_k\}$ for distinct $i,j,k \in \{1,2,3\}$, the factor in the Euler product for each prime $p=2,3,5$ will be identical. Therefore, fix such a $p$, and let $\Lambda_0$ be the maximal order in $\Q_p\B$. Up to ordering of the primitive  idempotents,  we may assume $\Lambda_0=\langle e_0, e_1, e_1+e_2\rangle$ and $\Lambda \coloneqq \Z_p\B=\langle e_0, pe_1, e_1+e_2\rangle$.

Then $\zeta_{\Z_p\B}(s)=Z_{\Lambda}(\Lambda, \Lambda_0; s)+Z_{\Lambda}(\Lambda, \Lambda;s)$, where as before

\[ Z_{\Lambda}(\Lambda, M; s)=\mu(\Aut M)^{-1} (\Lambda: M)^{-s} \int_{\Q\B^\times}  \Phi_{\{M:\Lambda\}}(x) \left\Vert x \right\Vert_V^s d^\times x,\] and in this case $\Aut \Lambda=\Lambda^\times$ and $ \Aut \Lambda_0=\Lambda_0^\times$.

We can compute that $\{\Lambda_0:\Lambda\}$ is a $\Z_p$-lattice with basis $\{e_0, pe_1, pe_2\}$, and of course $\{\Lambda:\Lambda\}=\Lambda$ with basis $\{e_0, pe_1, e_1+e_2\}$. Then 

\[ Z_{\Lambda}(\Lambda, \Lambda_0; s)=p^s  \left(\int_{\Z_p-\{0\}} (\Z_pa:\Z_p)^s d^\times a \right) \left(\int_{p\Z_p-\{0\}} (\Z_pb:\Z_p)^s d^\times b\right)^2= p^{-s}(1-p^{-s})^{-3} \] and

\[ Z_{\Lambda}(\Lambda, \Lambda; s)=(p-1)  \left(\int_{\Z_p-\{0\}} (\Z_pa:\Z_p)^s d^\times a \right) \left(\int_{p\Z_pe_1\oplus \Z_p(e_1+e_2)} (\tilde{\Lambda}y:\tilde{\Lambda})^s d^\times y\right),\] where $\tilde{\Lambda}$ is the maximal order in $\Q_pe_1 \oplus \Q_pe_2$.  Using  Equation (\ref{Lambda_mm-i}) in the case of $\Lambda_{10}$,
we obtain that 

\[ Z_{\Lambda}(\Lambda, \Lambda; s)=(p-1)\left( p^{-2s}(1-p^{-s})^{-3}+\frac{1}{p-1} (1-p^{-s})^{-1}\right)=(p^{1-2s}-2p^{-s}+1)(1-p^{-s})^{-3},\]
and therefore 

\[ \zeta_{\Z_p\B}(s)=(p^{1-2s}-p^{-s}+1)\zeta_{\Z_p}(s)^3.\]

So we can conclude that, for the integral adjacency algebra corresponding to the Petersen graph $P$, 
\begin{equation}\label{Petersengraph}
\zeta_{\Z P}(s) = \big( \prod_{p=2,3,5} (p^{1-2s}-p^{-s}+1) \big) \zeta_{\Z}(s).
\end{equation} 

\medskip
\noindent {\bf (b) The square.}  Let $\B=\{b_0,b_1,b_2\}$ be the standard basis of the association scheme corresponding to the square $S$, with elements of valencies $1$, $1$, and $2$, respectively.  As the primitive idempotents of $\Q\B$ are $e_0 = \frac{1}{4}(b_0+b_1+b_2)$, $e_1=\frac{1}{4}(b_0+b_1-b_2)$, and $e_2=\frac{1}{2}(b_0-b_1)$, an integral basis of $\Z\B$ is $\{4e_1,2e_2, b_0=e_0+e_1+e_2\}$.   This tells us the only prime relevant to the zeta function calculation is $p=2$.  

Let $\Z_2\B = \langle b_0, 4e_1, 2e_2 \rangle$.  We will use this notation to express $\Z_2$-lattices in this section.   The $\Z_2\B$-lattices between $\Lambda_0$ and $\Z_2\B$ are: $\Z_2\B$, $\langle e_0+e_1,4e_1,e_2 \rangle$, $\langle b_0, 2e_1, 2e_2 \rangle$, $\langle b_0, 2e_1+e_2, 2e_2 \rangle$, $\langle e_0+e_1,2e_1,e_2 \rangle$, $\langle e_0+e_2,e_1,2e_2 \rangle$, $\langle e_0,e_1+e_2,2e_2 \rangle$, and the maximal order $\Lambda_0 =\langle e_0,e_1,e_2 \rangle$.  By Lemma \ref{LatticesAbove} and Remark \ref{HermiteZeroes}, they give a complete set of representatives of isomorphism classes of $\Z_2\B$-lattices in $\Q \B$.  The index of $\Z_2\B$ in these lattices is $1$, $2$, $2$, $2$, $4$, $4$, $4$, and $8$, respectively.  To calculate the indices of unit groups of these orders, note that arbitrary units of $\Lambda_0$ are of the form $u_0e_0+u_1e_1+u_2e_2$ with $u_i \in 1 + 2\Z_2$ for $i=0,1,2$, so the index of a unit group of an order in $\Lambda_0^{\times}$ is increased by a factor of $2$ when an additional congruence condition mod $2^2$ is introduced on one of the $u_i$.  It follows that the unit groups of $\Z_2\B$ and $\langle e_0+e_1,4e_1,e_2 \rangle$ have index $2$ in $\Lambda_0^{\times},$ and the unit groups of the others have index $1$.   

Since the $\Z_2\B$-lattices $M$ under consideration all contain $b_0=1$, we always have $\{ M : \Z_2\B \} \subseteq \Z_2\B$, and may satisfy additional conditions.  If $M = \langle b_0,g_1,g_2 \rangle$, then the additional conditions arise from $g_1x \in \Z_2\B$ and $g_2x \in \Z_2\B$.  For example, if $x = x_0 b_0 + 4x_1e_1 + 2x_2e_2 \in \{ \langle b_0,4e_1,e_2 \rangle : \Z_2\B \}$, then $4e_1 x = 4(x_0+4x_1)e_1$ and $e_2x=(x_0+2x_2)e_2$ lie in $\Z_2\B$ if and only if $x_0 \in 2\Z_2$ and $x_1,x_2 \in \Z_2$.  Calculating all seven of the complementary lattices for our overorders in this way gives $ \{ \Z_2\B: \Z_2\B \} = \Z_2\B, $

\begin{align*}
 \langle 2b_0,4e_1,2e_2 \rangle &  =\{ \langle b_0,4e_1,e_2 \rangle : \Z_2\B \} =  \{ \langle b_0,2e_1,2e_2 \rangle : \Z_2\B \} \\
 & = \{ \langle b_0,e_1,2e_2 \rangle : \Z_2\B \} =   \{ \langle b_0,2e_1+e_2,2e_2 \rangle : \Z_2\B \} \mbox{, and }   \\
 \langle 4b_0, 4e_1, 2e_2 \rangle & = \{ \langle b_0, e_1+e_2,2e_2 \rangle : \Z_2\B \} = \{ \langle b_0,2e_1,e_2 \rangle : \Z_2\B \} = \{ \Lambda_0 : \Z_2\B \} . 
\end{align*}

\medskip
Let us write $\zeta$ for $(1-2^{-s})^{-1}$.  Now, $\langle 4b_0, 4e_1, 2e_2 \rangle = \langle 4e_0,4e_1,2e_2 \rangle$, so 
$$\displaystyle{\int_{A^{\times} \cap \langle 4e_0,4e_1,2e_2 \rangle}} ||x||_V^s d^{\times}x = \bigg(\displaystyle{\int_{4\Z_2-\{0\}}} (\Z_2a:\Z_2)^s d^{\times}a \bigg)^2 \bigg(\displaystyle{\int_{2\Z_2-\{0\}}} (\Z_2b: \Z_2)^s d^{\times}b \bigg) $$
$$ = (2^{-2s}\zeta)^2(2^{-s}\zeta) = 2^{-5s} \zeta^3. $$

Also $\langle 2b_0,4e_1,2e_2 \rangle = \langle 2e_0+2e_1,4e_1,2e_2 \rangle$.  From Section \ref{dim2}, with $B = \Q_2(e_0+e_1) + \Q_2e_1$ having maximal order $\tilde{\Lambda}_{00}=\Z_2[e_0+e_1, e_1]$, we saw that 
$$\displaystyle{\int_{B^{\times} \cap \tilde{\Lambda}_{01}}} ||y||_V^s d^{\times}y = (2^{-2s}(1-2^{-s})^{-2}+1).$$  
So calculating as in Section \ref{dim2} we get
$$\begin{array}{rcl} 
\displaystyle{\int_{A^{\times} \cap \langle 2b_0,4e_1,2e_2\rangle}} ||x||_V^s d^{\times}x &=& \bigg(\displaystyle{\int_{B^{\times} \cap \tilde{\Lambda}_{12}}} ||y||_V^2 d^{\times}y\bigg) \bigg(\displaystyle{\int_{2\Z_2-\{0\}}} (\Z_2b: \Z_2)^s d^{\times}b \bigg) \\
&=& \big( 2^{-2s}(2^{-2s} \zeta^2 +1) \big) \big(2^{-s} \zeta \big) \\
&=& (2^{1-5s}-2^{1-4s} + 2^{-3s})\zeta^3. 
\end{array}$$

\medskip
Finally, $\Z_2\B = \langle b_0,4e_1,2e_2 \rangle$.  This time we decompose the integral 
$\displaystyle{\int_{A^{\times} \cap \langle b_0,4e_1,2e_2 \rangle}} ||x||_V^s d^{\times}x$ over $2$ disjoint subsets, $\Z_2\B^{\times}$ and $A^\times \cap \langle 2b_0,4e_1,2e_2 \rangle$.
For $\langle b_0,4e_1,2e_2 \rangle^\times$, we use the fact that $||x||_V = (\Lambda_0 x : \Lambda_0) =1$ for $x \in \langle b_0,4e_1,2e_2 \rangle^{\times}$, so the integral in this case will be the measure $\mu(\langle b_0,4e_1,2e_2 \rangle^{\times}) = 2^{-1}$.  
The integral for the second subset has been computed above. 
Putting this all together, 
\begin{align*}
\displaystyle{\int_{A^{\times} \cap \langle b_0,4e_1,2e_2 \rangle}} ||x||_V^s d^{\times}x  & = 2^{-1}+(2^{1-5s}-2^{1-4s} + 2^{-3s})\zeta^3\\
& = \left( 2^{1-5s}-2^{1-4s}+2^{-1-3s}+3\cdot 2^{-1-2s}-3\cdot 2^{-1-s}+2^{-1} \right) \zeta^3.
\end{align*}

Therefore, the genus zeta functions we need for $\Lambda = \Z_2\B$ are: 
$$\begin{array}{l}
Z_{\Lambda}(\Lambda,\Lambda_0;s) = (1)(2^{3s})(2^{-5s})\zeta^3 = 2^{-2s} \zeta^3,  \\
Z_{\Lambda}(\Lambda, \langle b_0,2e_1,e_2 \rangle ;s) = (1)(2^{2s})(2^{-5s})\zeta^3 = 2^{-3s} \zeta^3,  \\
Z_{\Lambda}(\Lambda,\langle b_0,e_1+e_2,2e_2 \rangle ;s) = (1)(2^{2s})(2^{-5s})\zeta^3 = 2^{-3s} \zeta^3, \\
Z_{\Lambda}(\Lambda,\langle b_0,e_1,2e_2 \rangle;s) = (1)(2^{2s})(2^{1-5s}-2^{1-4s} + 2^{-3s})\zeta^3 = (2^{1-3s}-2^{1-2s}+2^{-s})\zeta^3, \\
Z_{\Lambda}(\Lambda,\langle b_0,2e_1,2e_2 \rangle ;s) = (1)(2^{s})(2^{1-5s}-2^{1-4s} + 2^{-3s})\zeta^3 = (2^{1-4s}-2^{1-3s}+2^{-2s})\zeta^3,\\
Z_{\Lambda}(\Lambda,\langle b_0,4e_1,e_2 \rangle;s) = (2)(2^{s})(2^{1-5s}-2^{1-4s} + 2^{-3s})\zeta^3 = (2^{2-4s}-2^{2-3s}+2^{1-2s})\zeta^3, \\
 Z_{\Lambda}(\Lambda,\langle b_0,2e_1+e_2,2e_2 \rangle;s) = (1)(2^{s})(2^{1-5s}-2^{1-4s} + 2^{-3s})\zeta^3 = (2^{1-4s}-2^{1-3s}+2^{-2s})\zeta^3, 
\end{array}$$
and 
\begin{align*}
Z_{\Lambda}(\Lambda,\Lambda ;s) & = (2)(1)\left( 2^{1-5s}-2^{1-4s}+2^{-1-3s}+3\cdot 2^{-1-2s}-3\cdot 2^{-1-s}+2^{-1} \right) \zeta^3 \\
 & = \big( 2^{2-5s}-2^{2-4s}+2^{-3s}+3\cdot 2^{-2s}-3\cdot 2^{-s}+1\big) \zeta^3.
\end{align*}

$\zeta_{\Z_2\B}(s)$ will be the sum of these eight genus zeta functions.
Adding them together, we get 
\[ \zeta_{\Z_2\B}(s) = (2^{2-5s}+ 2^{2-4s}-3\cdot 2^{-3s}+3\cdot2^{1-2s}-2^{1-s}+1) \zeta_{\Z_2}(s)^3, \]
and since $p=2$ is the only relevant prime, this determines $\zeta_{\Z S}(s)$.

\begin{thm}\label{Squaregraph}
The zeta function for the integral adjacency algebra $\Z S$ of the association scheme $S$ of rank $4$ corresponding to the square is: 
$$\zeta_{\Z S}(s) =(2^{2-5s}+ 2^{2-4s}-3\cdot 2^{-3s}+3\cdot2^{1-2s}-2^{1-s}+1) \zeta_{\Z}(s)^3.$$
\end{thm}

\begin{rmk} {\rm The square gives the first example the authors have computed for which the leading term of the functional equation polynomial differs in degree from the one shown to hold for group ring orders in \cite[pg. 158]{Bushnell-Reiner1980}. The authors do not currently know how to predict the degree and leading coefficient of this functional equation polynomial in these cases.  
%Further study will be needed to determine the nature of the functional equation for zeta functions of standard integral table algebras.  %It seems an interesting question to predict the degree and leading coefficient of this functional equation polynomial.
}  \end{rmk}

\medskip

\noindent {\bf (c) The generalized quadrangle $GQ(2,1)$.} Let $\B=\{b_0, b_1, b_2\}$ be the standard basis of the association scheme corresponding to $GQ(2,1)$ with intersection array $[4, 2; 1,2 ]$. The primitive idempotents of $\Q{\bf B}$ are:
$$ e_0 =\frac{1}{9}(b_0+b_1+b_2), \quad e_1 =\frac{1}{9}(4b_0-2b_1+b_2), \mbox{ and } e_2 = \frac{1}{9}(4b_0+b_1-2b_2). $$

We choose the integral basis $\{b_0, 3(e_1+2e_2),9e_2 \}$ for $\Z\B$. The only prime where $\Z \B$ is not maximal is at $p=3$.  Let $\Lambda \coloneqq \Z_3\B = \langle b_0, 3e_1+6e_2, 9e_2 \rangle$.

The $\Z_3$-lattices containing $\Z_3\B$ are $\Lambda_{000} = \langle e_0, e_1, e_2 \rangle$, $\Lambda_{0\!+\!01} = \langle e_0, e_1+e_2, 3e_2 \rangle$, $\Lambda_{01\!+\!0}=\langle e_0+e_1, 3e_1, e_2 \rangle$, $\Lambda_{0'1\!+\!0} =\langle e_0+e_2, e_1, 3e_2 \rangle$, $\Lambda_{011}= \langle e_0+e_1+e_2, 3e_1, 3e_2 \rangle$, and $\Lambda$, which are all overorders of $\Z_3\B$, as well as  $M=\langle e_0+2e_2, e_1+2e_2, 3e_2 \rangle$, which is not an overorder, but nevertheless a $\Lambda$-lattice. Lemma  \ref{LatticesAbove}  and Remark \ref{HermiteZeroes} imply we now have a complete set of representatives of isomorphism classes of $\Z_3\B$-lattices in $\Q \B$.

The nontrivial indices are given by 
$$[\Lambda_{000}:\Lambda]=3^3, [\Lambda_{0\!+\!01}:\Lambda] = [\Lambda_{01\!+\!0}:\Lambda]=[\Lambda_{0'1\!+\!0}:\Lambda]= [M:\Lambda]=3^2, \mbox{ and } [\Lambda_{011}:\Lambda]=3.$$  For the unit group measures we find 
$$\mu(\Lambda_{000}^\times)=1, \mu(\Lambda_{0\!+\!01}^\times)=\mu(\Lambda_{01\!+\!0}^\times)=\mu(\Lambda_{0'1\!+\!0}^\times)=\frac12, \mu(\Lambda_{011}^\times)= \frac14, \mbox{ and } \mu(\Lambda^\times) = \frac{1}{12},$$  as well as $\{M:M\}=\langle b_0, 3e_1, 3e_2\rangle  $ and thus $$\mu(\Aut M)=\frac{1}{4}.$$  We illustrate the type of calculations used to get the unit group measures by looking at $\mu(\Lambda^\times)$, where we can work modulo $27$. Then the units in $\Lambda_0 \pmod{27}$ are of the form $(u_1, u_2, u_3)$, where $u_i \in (\Z/27\Z)^\times$, and there are $18^3$ choices. On the other hand, the units in $\Lambda \pmod{27}$ are of the form $(u, u+3v, u+6v+9w)$, where $u \in (\Z/27\Z)^\times$, $v \in \Z/9\Z$, and $w \in \Z/3\Z$, so there are $18\cdot 9 \cdot 3$ choices, and the index is $(\Lambda_0^\times: \Lambda^\times)=12$.

Next, we compute the complementary $\Z_3$-lattices using standard linear algebraic techniques, and find:   
\begin{align*}
\{\Lambda_{000}:\Lambda\} =\langle 9e_0,9e_1,9e_2 \rangle = \Lambda_{222}, & & \{\Lambda_{0\!+\!01}:\Lambda \} = \langle 9e_0, 3e_1+6e_2, 9e_2\rangle,\, \\
\{\Lambda_{01\!+\!0}:\Lambda\} =\langle 3e_0+6e_1,9e_1,9e_2 \rangle, \, &  &\{\Lambda_{0'1\!+\!0}:\Lambda\} =\langle 3e_0+6e_2, 9e_1, 9e_2 \rangle, \\
 \{\Lambda_{011}:\Lambda\} = \langle 3b_0, 3e_1+6e_2, 9e_2 \rangle,  \,\,\,& &  \{M: \Lambda\}=\langle 3b_0, 9e_1, 9e_2\rangle, \qquad \quad  \,\,\,\,\,
\end{align*}
and  $\{ \Lambda : \Lambda \} = \langle b_0, 3e_1+6e_2, 9e_2 \rangle = \Lambda. $

Now we compute the local zeta integrals for these $\Z_3$-orders relative to $\Lambda$.  
First, since $\Lambda_{222} = \langle 9e_0, 9e_1, 9e_2 \rangle$, 

\begin{equation*}
\label{rank3max}
\displaystyle{\int_{A^{\times} \cap \Lambda_{222}}} ||x||_V^s d^{\times}x = \bigg(\displaystyle{\int_{9\Z_3-\{0\}}} (\Z_3a:\Z_3)^s d^{\times}a \bigg)^3 =\bigg( 3^{-2s} \zeta \bigg)^3= 3^{-6s} \zeta^3. 
\end{equation*}
and therefore
$$Z(\Lambda, \Lambda_{000}; s)= 3^{3s} \cdot 3^{-6s} \zeta^3 = 3^{-3s}\zeta^3. $$

For $Z(\Lambda, M; s)$ we integrate over $A^\times \cap \langle 3b_0, 9e_1, 9e_2\rangle$. If $x \in \{ M: \Lambda\}$, then   $x=3\gamma (e_0+e_1+e_2)+9\beta e_1+9\alpha e_2$  for some $\alpha, \beta, \gamma \in \Z_3$. Either $\gamma \in 3\Z_3$, or $\gamma \in u+3\Z_3$ for $u=1,2$. In the first case, $x\in \Lambda_{222}$. Otherwise, $x \in 3u(1+\Lambda_{111})$. Therefore,
$$\begin{array}{l}
\displaystyle{\int_{A^{\times} \cap \langle 3b_0,9e_1,9e_2 \rangle}} ||x||_V^s d^{\times}x =
\displaystyle{\int_{A^{\times} \cap \Lambda_{222}}} ||x||_V^s d^{\times}x + 
2 \bigg( \displaystyle{\int_{3(1+3\Z_3)-\{0\}}} ||a||^s d^\times a \bigg)^3 \\
\\
= \big( 3^{-6s} \zeta^3 \big) + 2 \big( 2^{-1} 3^{-s} \big)^3 =  (2^{-2} 3^{1-6s}+2^{-2} 3^{1-5s}-2^{-2} 3^{1-4s}+2^{-2} 3^{-3s})\zeta^3,
\end{array}$$   
and 
$$\begin{array}{rcl}
Z(\Lambda, M;s) &=& 4 \cdot 3^{2s} \cdot (2^{-2} 3^{1-6s}+2^{-2} 3^{1-5s}-2^{-2} 3^{1-4s}+2^{-2} 3^{-3s})\zeta^3 \\
&=& (3^{-s}-3^{1-2s}+3^{1-3s}+3^{1-4s} ) \zeta^3. \\
\end{array}$$

For $Z(\Lambda,\Lambda_{0'1\!+\!0};s)$, we need to integrate over $A^\times \cap \langle 3e_0+6e_2, 9e_1, 9e_2 \rangle$.    If $x = \gamma (3e_0+6e_2)+\beta (9e_1)+\alpha (9e_2)$ for some $\alpha, \beta, \gamma \in \Z_3$,  then either $\gamma  = 3\gamma'$, or $\gamma =  u + 3\gamma'$ for $u \in \{1,2\}$ and $\gamma'\in \Z_3$. In the first case,  $x \in \Lambda_{222}$.  Otherwise  $x=3(u+3\gamma')e_0+9\beta e_1+3(2u+3(\alpha+2\gamma'))e_2$ and $x \in u(3e_0+6e_2)+\Lambda_{222}$.   Therefore,   
$$\begin{array}{l}
\displaystyle{\int_{A^{\times} \cap \langle 3e_0+6e_2,9e_1,9e_2 \rangle}} ||x||_V^s d^{\times}x =
\displaystyle{\int_{A^{\times} \cap \Lambda_{222}}} ||x||_V^s d^{\times}x + 
2 \bigg(\displaystyle{\int_{A^{\times} \cap ((3e_0+6e_2)+\Lambda_{222})}} ||x||_V^s d^{\times}x \bigg) \\ 
 \\
= \big( 3^{-6s} \zeta^3 \big) + 2 \bigg(\displaystyle{\int_{9\Z_3-\{0\}}} ||b||^s d^{\times}b \bigg) \bigg( \displaystyle{\int_{3(1+3\Z_3)-\{0\}}} ||a||^s d^\times a \bigg) \bigg( \displaystyle{\int_{3(2+3\Z_3)-\{0\}}} ||c||^s d^\times c \bigg) \\
\\
= \big( 3^{-6s} \zeta^3 \big) + 2 \big( 3^{-2s} \zeta \big) \big( 2^{-1} 3^{-s} \big)^2 = \big( 3^{-6s} \zeta^3 \big) + 2^{-1} 3^{-4s} \zeta,
\end{array}$$
so 
\begin{equation}
\label{rank3nonmax1}
\displaystyle{\int_{A^{\times} \cap \langle 3e_0+6e_2,9e_1,9e_2 \rangle}} ||x||_V^s d^{\times}x = (2^{-1} 3^{-4s} - 3^{-5s} + 2^{-1} 3^{1-6s}) \zeta^3.
\end{equation}
Therefore, 
$$\begin{array}{rcl}
Z(\Lambda, \Lambda_{0'1\!+\!0};s) &=& 2 \cdot 3^{2s} \cdot (2^{-1} 3^{-4s} - 3^{-5s} + 2^{-1} 3^{1-6s}) \zeta^3  \\
&=& (3^{-2s} - 2\cdot 3^{-3s} +  3^{1-4s}) \zeta^3. \\
\end{array}$$

Since the integrals over $A^\times \cap \langle 9e_0, 3e_1+6e_2, 9e_2 \rangle$,  $A^\times \cap \langle 3e_0+6e_1, 9e_1, 9e_2 \rangle$, and $A^\times \cap \langle 3e_0+6e_2,9e_1,9e_2 \rangle$ will give the same result, we also have 
$$Z(\Lambda, \Lambda_{0'1\!+\!0};s) =Z(\Lambda, \Lambda_{01\!+\!0};s)= Z(\Lambda, \Lambda_{0\!+\!01};s) = (3^{-2s} - 2\cdot 3^{-3s} +  3^{1-4s}) \zeta^3. $$

For $Z(\Lambda,\Lambda_{011};s)$ our integral is over $A^\times \cap \langle 3b_0, 3e_1+6e_2, 9e_2 \rangle$.  There are two cases for $x = \gamma(3b_0) + \beta(3e_1+6e_2) + \alpha(9e_2)$ with $\alpha, \beta, \gamma \in \Z_3$: either $\gamma=3\gamma'$, or $\gamma=3\gamma'+u$ for some $u \in \{1,2\}$ and  $\gamma'\in \Z_3$.  In the first case, $x \in \langle 9b_0, 3e_1+6e_2, 9e_2 \rangle$, and we have computed the associated integral in Equation (\ref{rank3nonmax1}).  In the second case, we have $x = 3u + \gamma'(9b_0) + \beta(3e_1+6e_2) + \alpha (9e_2) \in 3u(1+3u\langle 3b_0, e_1+2e_2, 3e_2 \rangle)$, a multiplicative translate of $3(1+\langle 3b_0, e_1+2e_2, 3e_2 \rangle)$.  So 

\[\displaystyle{\int  \displaylimits_{A^{\times} \cap \langle 3b_0, 3e_1+6e_2,9e_2 \rangle}} \!\!\!\!\!\!\!\! ||x||_V^s d^{\times}x = \int  \displaylimits_{A^\times \cap \langle 9b_0, 3e_1+6e_2, 9e_2 \rangle} \!\!\!\!\!\!\!\! ||x||_V^s d^{\times}x + 2 \cdot \int  \displaylimits_{3(1+\langle 3b_0, e_1+2e_2,3e_2 \rangle)} \!\!\!\!\!\!\!\! ||x||_V^s d^{\times}x. \]

The first integral is already computed.   To calculate the second integral, note that $\langle 3b_0, e_1+2e_2,3e_2 \rangle = \langle 3e_0, e_1+2e_2,3e_2 \rangle$, so an element of $3(1+\langle 3b_0, e_1+2e_2,3e_2 \rangle)$ is of the form $3((1+3\alpha)e_0 + (1+\beta)e_1 + (1 + 2\beta + 3\gamma)e_2)$ for some $\alpha, \beta, \gamma \in \Z_3$.  Therefore, if   $L = 3((1+\Z_3)(e_1+2e_2)+(1+3\Z_3)e_2)=3((e_1+e_2)+\langle e_1+2e_2, 3e_2 \rangle)$ and $B$ is the subring $\langle e_1, e_2 \rangle$,  
then 
\[\int \displaylimits_{A^\times \cap 3(1+\langle 3b_0,e_1+2e_2,3e_2 \rangle)} \!\!\!\!\!\!\!\!\!\! ||x||_V^s d^{\times}x =\left(\int \displaylimits_{3(3\Z_3+1)}(\Z_3a:\Z_3)^sd^\times a\right) \left(\int \displaylimits_{\substack{B^\times \cap L}} ( B y: B)d^\times y \right).\]
To compute the second factor, we consider the three cases for an element $y=(3+3\beta)e_1+(3+6\beta+9\gamma)e_2\in L$,  depending on whether $\beta=3d, 3d+1$ or $3d+2$, and we can decompose the integral 

\begin{align*}
\int \displaylimits_{B^\times \cap L} ||x||_V^s d^{\times}x & =\int \displaylimits_{B^\times \cap \{18d+9\gamma+3, 9d+3\}}+\int \displaylimits_{B^\times \cap \{(18d+9\gamma+3, 9d+3)\}}+\int \displaylimits_{B^\times \cap \{(18d+9\gamma+6, 9d+6)\}}\\
 &= \int\displaylimits_{3(3\Z_3+1)}\int \displaylimits_{3(3\Z_3+1)}+ \int \displaylimits_{9\Z_3-\{0\}} \int \displaylimits_{3(3\Z_3+2)} +\int \displaylimits_{3(3\Z_3+2)} \int \displaylimits_{9\Z_3-\{0\}}\\
 &= (2^{-1} 3^{-s})^2+3^{-2s}\zeta \cdot 2^{-1} 3^{-s}+  2^{-1} 3^{-s}\cdot 3^{-2s}\zeta \\
& = \frac{1}{4}(3^{-2s}+3^{1-3s})\zeta.
\end{align*}
Therefore, 
\[\int  \displaylimits_{3(1+\langle 3b_0, e_1+2e_2,3e_2 \rangle)} ||x||_V^s d^{\times}x=2^{-1}3^{-s}\cdot 2^{-2} (3^{-2s}+3^{1-3s})\zeta =2^{-3}(3^{-3s}+3^{1-4s})\zeta.\]
So the integral we need for the last two cases works out to  
\[\displaystyle{\int  \displaylimits_{A^\times \cap \langle 3b_0, 3e_1+6e_2, 9e_2 \rangle}} \!\!\!\!\!\!\!\! ||x||_V^s d^{\times}x = \int  \displaylimits_{A^\times \cap \langle 9b_0,3e_1+6e_2,9e_2\rangle} \!\!\!\!\!\!\!\! ||x||_V^s d^{\times}x + 2 \!\!\!\!\!\!\!\!\int  \displaylimits_{3(1+\langle 3b_0,e_1+2e_2,3e_2\rangle)} \!\!\!\!\!\!\!\! ||x||_V^s d^{\times}x \]
\[ = (2^{-1} 3^{-4s} - 3^{-5s} + 2^{-1} 3^{1-6s}) \zeta^3+2^{-2}(3^{-3s}+3^{1-4s})\zeta \]
\begin{equation}
\label{rank3nonmax2}
= 2^{-2}(3^{2-6s}-3^{2-5s}+3^{1-4s}+3^{-3s}) \zeta^3
\end{equation}
This makes 
\begin{align*}
Z(\Lambda, \Lambda_{011}; s) & = 4\cdot 3^{s} \cdot 2^{-2}(3^{2-6s}-3^{2-5s}+3^{1-4s}+3^{-3s})\zeta^3 \\ 
& = ( 3^{2-5s}- 3^{2-4s}+3^{1-3s}+ 3^{-2s})\zeta^3.
\end{align*} 

For $Z(\Lambda,\Lambda;s)$ we need to integrate over $A^\times \cap \Lambda = A^\times \cap \langle b_0, 3e_1+6e_2, 9e_2 \rangle$.  If $x = \gamma b_0 + \beta(3e_1+6e_2) + \alpha (9e_2) \in \Lambda$  for some $\alpha, \beta, \gamma \in \Z_3$, then either $\gamma=3\gamma'$ for $\gamma' \in \Z_3$ or $\gamma=3\gamma'+u$ for some $u \in {1,2}$.  In the first case, $x \in \langle 3b_0,3e_1+6e_2,9e_2 \rangle$, and we have computed the associated integral in Equation (\ref{rank3nonmax2}).  In the second case, we have $x \in u(1 + u\langle 3b_0, 3e_1+6e_2, 9e_2 \rangle)$.  Since $1 + \langle 3b_0,3e_1+6e_2,9e_2 \rangle$ is a subgroup of $\Lambda_{000}^\times$ with measure $\mu(1+\langle 3b_0,3e_1+6e_2,9e_2 \rangle) = \frac{1}{24}$, we have 

\begin{align*}
\displaystyle{\int \displaylimits_{A^{\times} \cap \langle b_0,3e_1+6e_2,9e_2 \rangle}} ||x||_V^s d^{\times}x & = \int \displaylimits_{A^\times \cap \langle 3b_0,3e_1+6e_2,9e_2\rangle} ||x||_V^s d^{\times}x + 2 \int \displaylimits_{1+\langle 3b_0,3e_1+6e2,9e_2 \rangle} ||x||_V^s d^{\times}x \\
 & =  2^{-2}(3^{2-6s}-3^{2-5s}+3^{1-4s}+3^{-3s})\zeta^3+ 2 \cdot 2^{-3}\cdot 3^{-1}\\
& = \frac{1}{12} (3^{3-6s}-3^{3-5s}+3^{2-4s}+2\cdot 3^{-3s}+3^{1-2s}-3^{1-s}+1)\zeta^3.
\end{align*}
So 
\[ Z(\Lambda, \Lambda; s) = (3^{3-6s}-3^{3-5s}+3^{2-4s}+2\cdot 3^{-3s}+3^{1-2s}-3^{1-s}+1)\zeta^3.   \]
 
The zeta function we want is obtained from the sum of the the local zeta integrals obtained above. %After finding the sum of these Bushnell-Reiner integrals 

\begin{thm}\label{GenQuadranglegraph}
The zeta function of the integral adjacency algebra of the association scheme corresponding to the generalized quadrangle graph $GQ(2,1)$ is: 
$$\zeta_{\Z GQ(2,1)}(s)=(3^{3-6s}-2\cdot  3^{2-5s}+4\cdot 3^{1-4s}+3^{1-3s}+ 4\cdot 3^{-2s}- 2\cdot 3^{-s}+1)\zeta_{\Z}(s)^3.$$
\end{thm}

\bigskip

\begin{rmk}  The recursive method we have employed throughout can in theory be used to compute the zeta functions for commutative local orders of any rank.  Lemma \ref{LatticesAbove} and Remark \ref{HermiteZeroes} will need to be adapted to the non-split situation.  After that the most difficult step in the calculation is the dividing of the local zeta integral into parts, some of which have already been computed, and others that are sums of multiplicative translates of ones that are straightforward to compute.  In the above examples, the number of these translates depended upon a single congruence.  As the rank and complexity increases, this decomposition step will become more complicated, and better techniques for decomposing the integral and managing the pieces will be needed.
\end{rmk}  

\medskip
{\bf Acknowledgement:} The authors would like to thank the anonymous referee for several insightful comments that significantly improved the exposition of this paper.


\begin{thebibliography}{10}

\bibitem{AFM} Z.~Arad, E.~Fisman, and M.~Muzychuk, Generalized table algebras, {\it Israel J. Math.}, {\bf 114} (1999), 29-60. 

\bibitem{DRdotorg} R.~F.~Bailey, {\it DistanceRegular.org}, Memorial University, 2019. {\tt http://www.distanceregular.org}

\bibitem{Bannai-Ito84} E. Bannai and T. Ito, {\it Algebraic Combinatorics I: Association Schemes}, Benjamin Cummings, Menlo Park, CA, 1984. 

\bibitem{Blau09} H.~Blau, Table algebras, {\it European J. Combin.}, {\bf 30} (2009), 1426-1455.

\bibitem{Bushnell-Reiner1980} C. J. Bushnell and I. Reiner, Zeta functions of arithmetic orders and Solomon's conjectures, {\it Math. Z.}, {\bf 173} (2), (1980),  135--161. 

%\bibitem{EGNO} P.~Etingof, S.~Gelaki, D.~Nikshych, and V.~Ostrik, {\it Tensor categories}, Mathematical Surveys and Monographs, 205, American Mathematical Society, Providence, RI, 2015. 

\bibitem{Caruso} X. Caruso, Computations with $p$-adic numbers, {\it Journ{\'e}es Nationales de Calcul Formel}, Les cours du CIRM, {\bf{5}} (2017), 1–75.

\bibitem{Jenner1963} W.~Jenner, Zeta functions of non-maximal orders in rational semisimple algebras, {\it Duke Math. J.}, {\bf 30} (1963), 541–543.

%\bibitem{Hanaki-Reps} A. Hanaki, Representations of Association Schemes, {\it European J. Combin.}, {\bf 30} (2009), 1477-1496.

\bibitem{Hanaki-Hirasaka2016} A. Hanaki and M. Hirasaka, Zeta functions of adjacency algebras of association schemes of prime order or rank two, {\it Hokkaido J. Math.}, {\bf 45} (1), (2016), 75-91.  

%\bibitem{Herman2020} A. Herman,  Schur indices for reality-based algebras with two nonreal basis elements, {\it J. Algebra Appl.}, to appear. {\tt https://doi.org/10.1142/S0219498821501772}

\bibitem{Herman-Singh2018} A. Herman and G. Singh, Orders of torsion units of integral reality-based algebras with rational multiplicities, {\it J. Algebra Appl.}, {\bf 17} (1), (2018), 1850015 (12 pgs)

\bibitem{Herman-Hirasaka-Oh2017} A. Herman, M. Hirasaka, and Semin Oh, Zeta functions for tensor products of locally coprime integral adjacency algebras of association schemes, {\it Comm. Algebra},  {\bf 45} (11),  (2017), 4896–4905. 

\bibitem{Hirasaka-Oh2018} M. Hirasaka and Semin Oh, The number of ideals of $\Z[x]$ containing $x(x-\alpha)(x-\beta)$ with given index, {\it J. Algebra}, {\bf 493} (2018), 36-56.

\bibitem{Hironaka85} Y. Hironaka, Corrections to my paper: ``Zeta functions of integral group rings of metacyclic groups'' [Tsukuba J. Math. 5 (1981), no. 2, 267–283; MR0653121], {\it Tsukuba J. Math.}, {\bf 9} (2), (1985), 373–374. 

\bibitem{Hironaka81} Y. Hironaka, Zeta functions of integral group rings of metacyclic groups, {\it Tsukuba J. Math.}, {\bf 5} (2), (1981), 267–283. 

\bibitem{Hofman2016} T. Hofmann, Zeta functions of lattices of the symmetric group, {\it Comm. Algebra}, {\bf 44} (5), (2016), 2243–2255. 

\bibitem{Kleinert1994} E. Kleinert, Units of classical orders: a survey, {\it L'Enseignement Math.}, {\bf 40} (1994), 205-248.
 
%\bibitem{L} G.~Lusztig, Leading coefficients of character values of Hecke algebras, Proc. Sympos. Pure Math., vol. 47, Amer. Math. Soc., Providence, RI, 1987, pp. 235–262, Part II.

\bibitem{Rossmann2018} T. Rossmann, Computing local zeta functions of groups, algebras, and modules, {\it Trans. Amer. Math. Soc.}, {\bf 370} (7), (2018), 4841–4879. 

\bibitem{Solomon77} L. Solomon, Zeta functions in integral representation theory, {\it Adv. Math.}, {\bf 26} (1977), 306-326. 

\bibitem{Takegahara87} Y. Takegahara, Zeta functions of integral group rings of abelian $(p,p)$-groups, {\it Comm. Algebra} {\bf 15} (12), (1987), 2565–2615.

\bibitem{Wittmann2004} C. Wittmann, Zeta functions of integral representations of cyclic $p$-groups, {\it J. Algebra}, {\bf 274} (2004), 271-308. 

\end{thebibliography}
\end{document}